\documentclass[12pt]{amsart}
\usepackage{amsmath,amssymb,amsfonts,latexsym,amscd,psfrag,mathabx,graphicx,mathrsfs,stmaryrd}
\usepackage[square,sort,comma,numbers]{natbib}
\usepackage{graphicx,tikz,tkz-berge,tkz-graph}
\usepackage{mathtools}
\usepackage[utf8]{inputenc}
\usepackage[T1]{fontenc}
\usepackage{lmodern}

\usepackage{tikz}
\usepackage{verbatim}
\usetikzlibrary{shapes,shadows,calc}
\usepgflibrary{arrows}
\usetikzlibrary{arrows, decorations.markings, calc, fadings, decorations.pathreplacing, patterns, decorations.pathmorphing, positioning}
\tikzset{nodc/.style={circle,draw=blue!50,fill=pink!80,inner sep=4.2pt}}
\tikzset{nod1/.style={circle,draw=black,fill=black,inner sep=1pt}}
\tikzset{noddee/.style={circle,draw=black,fill=black,inner sep=1.6pt}}
\tikzset{nod2/.style={circle,draw=black,fill=black,inner sep=2pt}}
\tikzset{nod3/.style={circle,draw=black,fill=black,inner sep=3pt}}
\tikzset{nodempty/.style={circle,draw=black,inner sep=2pt}}
\tikzset{nodel/.style={circle,draw=black,inner sep=2.2pt}}

\tikzset{nodinvisible/.style={circle,draw=white,inner sep=2pt}}
\tikzset{nodpale/.style={circle,draw=gray,fill=gray,inner sep=1.6pt}}

\usepackage{scalefnt}
\usetikzlibrary{arrows,decorations.pathmorphing,backgrounds,positioning,fit,petri}

\usepackage{subfigure}
\usepackage{float}
\usepackage[colorlinks, urlcolor=blue, linkcolor=blue, citecolor=blue]{hyperref}
\usepackage{cleveref}
\theoremstyle{plain}
\newtheorem{proposition}[equation]{Proposition}

\newtheorem{theorem}[equation]{Theorem}

\newtheorem{corollary}[equation]{Corollary}
\newtheorem{lemma}[equation]{Lemma}

\theoremstyle{definition}

\newtheorem{definition}[equation]{Definition}
\newtheorem{remark}[equation]{Remark}
\newtheorem{example}[equation]{Example}
\newtheorem{question}[equation]{Question} 
\newtheorem{fact}[equation]{Fact} 
\usepackage[colorlinks, urlcolor=blue, linkcolor=blue, citecolor=blue]{hyperref}
\newcommand{\Q}{\mathbb{Q}}

\newcommand{\SA}{\mathbb{S}}

\newcommand{\BE}{\mathfrak{B}}

\newcommand{\CE}{\mathcal{C}}

\newcommand{\RE}{\mathcal{R}}

\newcommand{\HE}{\mathcal{H}}

\newcommand{\NE}{\mathcal{N}}

\newcommand{\kk}{\Bbbk}
\newcommand{\hf}{\mathfrak{h}}

\newcommand{\Ind}{\operatorname{Ind}}
\newcommand{\pd}{\operatorname{prod-dim}}

\newcommand{\IR}{\operatorname{IR}}
\newcommand{\cd}{\operatorname{cochord}}
\newcommand{\reg}{\operatorname{reg}}

\newcommand{\girth}{\operatorname{girth}}

\newcommand{\im}{\operatorname{im}}
\newcommand{\m}{\operatorname{m}}

\newcommand{\degen}{\operatorname{degen}}

\newcommand{\IN}{\operatorname{Inci}}

\newcommand{\CNS}{\operatorname{CNS}}
\newcommand{\red}{\operatorname{red}}

\newcommand{\D}{\Delta}

\begin{document}
\title[Projective dimension of closed neighborhood hypergraphs]{Projective dimension of closed neighborhood hypergraphs via extended double covers}

\author{Yusuf Civan and Anurag Singh}
\address{Department of Mathematics, S{\"{u}}leyman Dem{\.{\.i}}rel University,
Isparta, 32260, T{\"{u}}rk{\.{\.i}}ye.}
\email{yusufcivan@sdu.edu.tr}
\address{Indian Institute of Technology, Bhilai, Chhattisgarh-491002, India.}
\email{anurags@iitbhilai.ac.in}
	
\date{\today}
%\thanks{}
	
\begin{abstract}
Let $G$ be a finite and simple graph without isolated vertices. We investigate the projective dimension of the closed neighborhood hypergraph $\NE[G]$ and its relationship with the Castelnuovo-Mumford regularity of the extended bipartite double cover $\BE_e(G)$ of $G$. We establish the general upper bound $\pd(\NE[G])\leq \reg(\BE_e(G))$ for all graphs. Furthermore, we prove that the exact equalities $\pd(\NE[G])=\reg(\BE_e(G))=\alpha(G)$ hold when $G$ belongs to several prominent graph classes, including K\"onig-Egerv\'ary (contains all bipartite graphs), cographs, co-chordal, chordal and comparability graphs, where $\alpha(G)$ denotes the independence number. Our method of proofs relies on connecting algebraic invariants to the underlying combinatorial structure of graphs through covering, domination and matching parameters, together with the use of homology tools.
\end{abstract} 
\maketitle
	
%%%%%%%%%%%%%%%%%%%%%%%%%%%%%%%%%%%%%%%%%%%%%%%%%%%%%%%%%%%%
\section{Introduction}\label{sect:intro}
A \emph{hypergraph} $\HE$ on a vertex set $V=\{x_1,\ldots,x_n\}$ is a family of subsets (edges) of $V$.  We identify the set of edges with the hypergraph $\HE$ itself. When $R=\kk[x_1,\ldots,x_n]=\bigoplus_{t=0}^{\infty} R_t$ is a polynomial ring over a field $\kk$ with the standard grading, the \emph{edge ideal} $I(\HE)$ of the hypergraph $\HE$ is defined to be the ideal of $R$ generated by all squarefree monomials $x_e:=\prod_{x_i\in e} x_i$ such that $e\in \HE$.
	
Consider the minimal free graded resolution of $R/I(\HE)$ as an $R$-module:
\begin{equation*}
0\to \bigoplus_j R(-j)^{\beta_{t,j}}\to \ldots \to \bigoplus_j R(-j)^{\beta_{1,j}}\to R\to R/I(\HE)\to 0.
\end{equation*}
The \emph{Castelnuovo-Mumford regularity} (simply the \emph{regularity}) and the \emph{projective dimension} of $R/I(\HE)$ are defined as
\begin{align*}
&\reg_{\kk}(\HE):=\reg_{\kk}(R/I(\HE))=\max \{j-i\colon \beta_{i,j}\neq 0\},\\
&\pd_{\kk}(\HE):=\pd_{\kk}(R/I(\HE))=\max \{i\colon \beta_{i,j}\neq 0\}
\end{align*}
respectively\footnote{Unless otherwise stated, our results are independent of the characteristic of the coefficient field. So, wherever it is appropriate, we drop $\kk$ from our notation.}.
	
Throughout we mainly consider hypergraphs associated to graphs in a certain manner. Let $G=(V,E)$ be a (finite and simple) graph without any isolated vertex. The hypergraph defined by 
\begin{displaymath}
\NE[G]:=\{N_G[x]\colon x\in V\}
\end{displaymath} 
is known as the \emph{closed neighborhood hypergraph} of $G$. These hypergraphs have  already appeared in different contexts. It is known that the recognition of neighborhood hypergraphs is polynomially equivalent to the matrix symmetrization problem~\citep{boros2008}. On the other side, the independence complex 
$\Ind(\NE[G])$ is known as the \emph{dominance complex} of the underlying graph $G$, and the (homological) connectivity of these complexes are well-studied~\citep{matsushita2022,matsushita2023}.  
	
The algebraic properties of closed neighborhood hypergraphs have been explored so far through their associated squarefree monomial ideals~\citep{sharifan2020}. 
Recently, Chakraborty et al.~\citep{anurag2024} establish exact combinatorial descriptions for the regularity of specific graph classes, while Joseph, Roy, and Singh~\citep{anurag2025} utilize Barile-Macchia resolutions to analyze their minimal free resolutions. Beyond projective dimension and regularity calculations, the determination of graph classes for which these ideals are Cohen-Macaulay also receives an important attention ~\citep{moradi2026, honeycutt, leaman}. Additionally, various other structural features such as their associated primes, normality, and asymptotic behavior have been investigated~\citep{hien, nasernejad1, nasernejad2, nasernejad3}.	
	
The primary concern of our current work is to compute the projective dimension of closed neighborhood hypergraphs in terms of known parameters of the underlying graphs. Our main approach is based on the verification that the projective dimension of closed neighborhood hypergraphs can be bounded from above by the regularity of bipartite graphs constructed from the underlying graphs. In fact, the more is true. In a recent work, B{\i}y{\i}ko{\u{g}}lu et al.~\cite{yusuf} introduce a bipartite graph (the bipartite complement of the facet-vertex bipartite graph) $B(\D)$ associated to any arbitrary simplicial complex $\D$, and prove that the regularity of $\D$ can be bounded from above by that of the independence complex of $B(\D)$ (compare to~\citep[Theorem~$3$]{yusuf}). When $\D=\Ind(\NE[G])^{\vee}$, the Alexander dual of the independence complex of the closed neighborhood hypergraph of a graph $G$ without any isolated vertex, the bipartite graph $B(\D)$ turns out to be a known bipartite transform of $G$, that was introduced by Alon~\cite{NA}.

Let $G$ be a graph on the vertex set $V=\{v_1,v_2,\ldots,v_n\}$. Then the \emph{extended double cover} $\BE_e(G)$ of $G$ is the bipartite graph on $V\cup V'$, where $V'=\{v'_1,v'_2,\ldots,v'_n\}$, in which $v_i$ and $v'_j$ are adjacent if and only if $i=j$ or $v_i$ and $v_j$ are adjacent in $G$. 
	
\begin{theorem}\label{thm:thm1}
The inequality $\pd(\NE[G])\leq \reg(\BE_e(G))$ holds for every graph $G$.
\end{theorem}
	
We note that the inequality of Theorem~\ref{thm:thm1} could be strict in general. In fact, the gap $\reg(\BE_e(G))-\pd(\NE[G])$ could be arbitrarily large.
	
Since closed neighborhood hypergraphs are grounded over graphs, the structural properties and invariants of the underlying graphs play a distinguished role on the determination of their algebraic properties. In particular, the relevant covering,  independence, matching and domination parameters will be essential invariants both for calculations and possible bounds on respective algebraic parameters of associated edge rings of these hypergraphs. While exact homological formulas have been recently discovered for certain classes, our bipartite translation allows us to prove a unifying sequence of exact equalities across a substantially broader array of graph families. In this direction, the main contribution of this paper is summarized in the following result.
	
\begin{theorem}\label{thm:thm2}
The equalities $\pd(\NE[G])=\reg(\BE_e(G))=\alpha(G)$ hold provided that $G$ belongs to any of the following classes:
\begin{enumerate}
\item  K{\"{o}}nig-Egerv{\'{a}}ry graphs (include all bipartite graphs) (Theorem~\ref{thm:konig_exact}),
\item cographs (Theorem~\ref{thm:cograph_exact}),
\item co-chordal graphs (Theorem~\ref{thm:cochordal_exact}),
\item chordal graphs (Theorem~\ref{thm:chordal_exact}),
\item comparability graphs (Theorem~\ref{thm:comparability}).
\end{enumerate}
\end{theorem}

We point out that all graph classes (except possibly non-bipartite K{\"{o}}nig-Egerv{\'{a}}ry graphs) appearing in Theorem~\ref{thm:thm2} have a common property, each of these classes constitutes a subclass of strongly perfect graphs (see Section~\ref{sec:prep} for the relevant definition). Thus, one may ask whether similar equalities hold for the broad class of strongly perfect graphs. On the one hand, we believe that a positive answer will not be a surprise. On the other hand, the method of proofs presented in this paper is not sufficient enough to resolve the general case of strongly perfect graphs (compare to Example~\ref{example:cd-alpha}). 

\begin{remark}
In an earlier unpublished work~\citep{BC-prod}, B{\i}y{\i}ko{\u{g}}lu and the first author have claimed~\citep[Theorem $3.23$]{BC-prod} that the equality $\pd(\NE[G])=\Gamma(G)$ holds for every graph $G$, where $\Gamma(G)$ denotes the upper domination number of $G$.  However, the proof of their result was dependent on~\citep[Theorem $5.4$]{ha-russ}, which was later shown to be wrong (compare to~\citep[Remark $8$]{Civan2022}). In particular, the equality $\pd(\NE[G])=\Gamma(G)$ does not need to hold as demonstrated by means of many examples in the present work. 
\end{remark}

%%%%%%%%%%%%%%%%%%%%%%%%%%%%%%%%%%%%%%%%%%%%%%%%%%%%%%%%%%%%%%%%
\section{Preliminaries}\label{sec:prep}
We first recall some general notions and notations needed throughout the
paper, and repeat some of the definitions mentioned in the introduction more
formally.

A \emph{hypergraph} $\HE$ is a family of nonempty subsets of a (finite) set $V$. The elements of $V$  and $\HE$ are vertices and edges of $\HE$. A vertex $v\in V$ is \emph{isolated} in $\HE$ if there exists no edge $e\in \HE$ with $v\in e$. A hypergraph $\HE$ is $k$-\emph{uniform} if $|e|=k$ for every $e\in \HE$. For instance, graphs are $2$-uniform hypergraphs. A hypergraph $\HE$ is said to be \emph{simple} (or \emph{Sperner}) if $e,f\in \HE$ with $e\subseteq f$, then $e=f$. 
	
By a (simple) graph $G=(V,E)$, we mean a finite undirected graph without loops or multiple edges. If $U\subset V$, the graph induced on $U$ is written $G[U]$, and in particular, we abbreviate $G[V\setminus\{x\}]$ to $G-x$.
	
When $G=(V,E)$ is a graph, the open and closed neighborhood of a vertex $v$ are $N_G(v)=\{u\in V\colon uv\in E\}$ and $N_G[v]=N_G(v)\cup \{v\}$, respectively. In particular, we set $N_G[A]:=\bigcup\limits_{a\in A}N_G[a]$.
	
A subset $I\subseteq V$ is \emph{independent} if no pair of vertices of $I$ forms an edge in $G$. The \emph{independence number} $\alpha(G)$ of $G$ is the maximum size of an independent set in $G$. Similarly, a subset $K\subseteq V$ is a clique of $G$ if every pair of vertices in $K$ forms an edge in $G$. The maximum size $\omega(G)$ of a clique in $G$ is the \emph{clique number} of $G$. 
	
We write bipartite graphs in the form $B=(X,Y;E)$ in which the vertex set $V(B)$ has the corresponding bipartition $X\cup Y$. Throughout $K_n$, $K_{n,m}$, $P_n$ and $C_k$ will denote the complete, complete bipartite, path and cycle graphs on $n,m\geq 1$ and $k\geq 3$ vertices respectively. In particular, $K_{1,m}$ is known as the \emph{star graph}. The vertex of $K_{1,m}$ of degree $m$ is said to be its \emph{center vertex}. A \emph{double-star graph} $DS_{n,m}$ is a graph obtained from the disjoint union of two star graphs $K_{1,n}$ and $K_{1,m}$ by adding an edge between their center vertices. That additional edge is called the \emph{center edge} of the double-star $DS_{n,m}$. In particular, we simply write $DS_n$ instead of $DS_{n,n}$ when $n=m$.

For a hypergraph $\HE$, the \emph{incidence graph} $\IN(\HE)$ of $\HE$ is the bipartite graph on $V\cup \HE$ such that $(v,e)\in E(\IN(\HE))$ if and only if $v\in e$. 
	
A matching in a graph is a subset of edges no two of which share a vertex. The maximum size $\m(G)$ of a matching in $G$ is the \emph{matching number} of $G$. An induced matching of a graph $G$ is a matching $M$ if no two vertices belonging to different edges of $M$ are adjacent in $G$. The maximum size of an induced matching of $G$ is known as the \emph{induced matching number} $\im(G)$ of $G$. 
	
We say that $G$ is $H$-free if no induced subgraph of $G$ is isomorphic to $H$. A graph $G$ is called \emph{chordal} if it is $C_r$-free for all $r>3$. Note that there are various alternative characterizations of chordal graphs~\citep{BLS-gclass}. Furthermore, recall that a vertex $v$ is said to be a \emph{simplicial vertex} if $N_G[v]$ is a clique of $G$. It is known that every chordal graph contains at least one simplicial vertex.

A graph $G$ is said to be \emph{co-chordal} if its complement $\overline{G}$ is a chordal graph. A \emph{co-chordal covering} of a graph $G$ is a family of co-chordal subgraphs of $G$ whose edge sets cover the edge set of $G$. The minimum size of such a covering is called the \emph{co-chordal covering number} of $G$, denoted by $\cd(G)$.
	
\begin{lemma}\textnormal{\citep{MK-im,russ}}\label{lem:cochord}
$\im(G)\leq \reg(G)\leq \cd(G)$ for every graph $G$.
\end{lemma}
	%%%%%%%%%%%%%%%%%%%%%%%%%%%%%%%%%%%%%%%%%%%%%%%%%%%%%%%%%%%%%%%%
\subsection{Domination parameters}
A subset $A$ of vertices is said to \emph{dominate} a set $B$ if $B\subseteq N_G[A]$. The minimum size of a set of vertices dominating a set $B$ is denoted by $\gamma_G(B)$. In particular, a subset $D\subseteq V$ is a \emph{dominating set} of $G$ if $V=N_G[D]$. A dominating set is \emph{minimal} if none of its proper subsets is a dominating set in $G$. The \emph{domination number} $\gamma(G)$ and the \emph{upper domination number} $\Gamma(G)$ are defined to be the minimum and maximum size of minimal dominating sets in $G$ respectively. 
	
Let $R\subseteq V$ be a subset of vertices. Then $R$ is called an \emph{irredundant set} if 
\begin{equation*}
N_G[x]-N_G[R-\{x\}]\neq \emptyset
\end{equation*}
for each $x\in R$. Let $\IR(G)$ denote the maximum
cardinality of an irredundant set in $G$. We note that the inequalities $\alpha(G)\leq \Gamma(G)\leq \IR(G)$ hold for every graph $G$.
	
We remark that the gap between each pair of $\alpha(G)$, $\Gamma(G)$ and $\IR(G)$ could be arbitrarily large. However, on certain graph classes they coincide. Recall that a graph $G$ is said to be a \emph{strongly perfect graph}~\citep{berge1984} if for every subset $S\subseteq V$, there exists an independent set $R\subseteq S$ such that $|R\cap C|=1$ for every maximal complete subgraph of $G[S]$. Observe that the class of strongly perfect graphs include chordal and co-chordal graphs, bipartite graphs, cographs, comparability graphs, permutation graphs (see~\citep[Section $3$]{dombook} for further details).
	
\begin{theorem}\label{thm:dom-eq}
If $G$ is a graph without any isolated vertex, then $\alpha(G)=\Gamma(G)=\IR(G)$ if one of the following holds:
\begin{itemize}
\item[$(1)$] $G$ is a strongly perfect graph~\textnormal{\citep[Theorem $2$]{Cheston1994}},

\item[$(2)$] $\IR(G)+\gamma(G)=|G|$~\textnormal{\citep[Theorem $2$]{Cockayne1981}}.
\end{itemize}
\end{theorem}
%%%%%%%%%%%%%%%%%%%%%%%%%%%%%%%%%%%%%%%%%%%%%%%%%%%%%%%%%%%%%%%%%%%%%%%%%%	
\subsection{Topological tools}
To analyze the algebraic properties of hypergraphs, we often rely on a few basic topological concepts. Recall that an \emph{abstract simplicial complex} $\Delta$ on a vertex set $V$ is a collection of subsets of $V$ that is closed under inclusion. The elements of $\Delta$ are its \emph{faces}, and faces which are maximal with respect to the inclusion are \emph{facets} of $\Delta$. The \emph{Alexander dual} of a simplicial complex is the complex defined by
\begin{equation*}
\Delta^{\vee}:=\{A\subseteq V\colon V\setminus A\notin \Delta\}.
\end{equation*}
Recall that a subset $K\subseteq V$ is said to be a minimal non-face of $\D$, if $K$ is itself not a face of $\D$ while any proper subset of it is. The maximum size of minimal non-faces of $\D$ is known as the \emph{Helly number} of $\D$, denoted by $\hf(\D)$. Observe that the inequality $\hf(\D)\leq \reg(\D)+1$ holds, since if $K$ is a minimal non-face in $\D$ of size $\hf(\D)$, then $\D[K]\cong \SA^{\hf(\D)-2}$, where $\SA^m$ is the $m$-dimensional sphere and $\D[K]:=\{A\in \D\colon A\subseteq K\}$ is the subcomplex of $\D$ induced by the set $K$.

Let $\HE$ be a hypergraph on the vertex set $V$. A subset $I\subseteq V$ is \emph{independent} if $e\nsubset I$ for each $e\in \HE$. The family $\Ind(\HE)$ of all independent sets in $\HE$ forms a simplicial complex, the \emph{independence complex} of $\HE$.

The regularity and projective dimension of the edge ring of a hypergraph $\HE$ is closely tied to the topology of $\Ind(\HE)$ via Hochster's classic formula~\citep{stanley-green}. In simple terms, Hochster's formula shows that the graded Betti numbers of the edge ring can be computed from the dimensions of the reduced homology groups of its induced subcomplexes. In other words,
\begin{align*}
&\reg(\HE)=\max \{j\colon \widetilde{H}_{j-1}(\Ind(\HE)[S];\kk)\neq 0\;\textrm{for\;some}\;S\subseteq V\},\\
&\pd(\HE)=\min \{i\colon \widetilde{H}_{|L|-i-l-1}(\Ind(\HE)[L];\kk)=0\;\textrm{for\;all}\;l>0\;\textnormal{and}\;L\subseteq V\},
\end{align*}
where $\widetilde{H}_{*}(-;\kk)$ denotes the (reduced) singular homology.
This useful homological bridge allows us to translate purely algebraic invariants into structural combinatorial properties of hypergraphs~\citep{ha-survey}. Moreover, the notion of the Alexander dual plays a complementary role via Terai's duality~\citep{Terai} on the calculation of the projective dimension of a hypergraph $\HE$ as the equality $\pd(\HE)=\reg(\Ind(\HE)^{\vee})+1$ holds.

We further note that the Helly number $\hf(\Ind(\HE)^{\vee})$ of the Alexander dual of $\Ind(\HE)$ exactly equals to $|V|-i(\HE)$, where $i(\HE)$ is the least cardinality of a maximal independent set of $\HE$, since any facet $F$ of $\Ind(\HE)$ corresponds to a minimal non-face $V\setminus F$ of $\Ind(\HE)^{\vee}$. This in particular implies the following lower bound by Terai's duality~\citep{Terai}.

\begin{corollary}\label{cor:helly}
The inequality $|V|-i(\HE)=\hf(\Ind(\HE)^{\vee})\leq \pd(\HE)$ holds for any hypergraph $\HE$.
\end{corollary}	
	
In the case of graphs, the following provides an inductive bound on the regularity, which will serve as a cornerstone for establishing exact regularity equalities across aforementioned graph classes. 

\begin{corollary}\textnormal{\cite{DHS-boundpd}}\label{cor:induction-sc}
Let $G$ be a graph and let $v\in V$ be given. Then
\begin{equation*}
\reg (G)\leq \max\{\reg(G-v), \reg(G-N_G[v])+1\}.
\end{equation*}
Moreover, $\reg(G)$ always equals to one of $\reg(G-v)$ or $\reg(G-N_G[v])+1$. 
\end{corollary}

%%%%%%%%%%%%%%%%%%%%%%%%%%%%%%%%%%%%%%%%%%%%%%%%%%%%%%%%%%%%%%%%%%%%%%%%%%
\section{Proof of Theorem~\ref{thm:thm1} and examples}\label{sec:pd_reg}
In this section, we prove Theorem~\ref{thm:thm1}, and provide examples to analyze its tightness with respect to the related graph parameters. 	
	
Let $G=(V,E)$ be a graph without any isolated vertex. Observe that $\NE[G]$ is not always a simple hypergraph. However, we may turn it into a simple hypergraph without altering its regularity and projective dimension.
\begin{definition}
The \emph{reduced closed neighborhood hypergraph} $\NE_{\red}[G]$ of $G$ is defined to be the hypergraph on $V$ whose edges are precisely the inclusion-minimal members of $\{N_G[x]\colon x\in V\}$. 
\end{definition}
	
\begin{fact}\label{fact:remove}
Let $\HE$ be a hypergraph. If $e,f\in \HE$ are two edges with $e\subseteq f$, then $\reg(\HE)=\reg(\HE-f)$ and $\pd(\HE)=\pd(\HE-f)$.
\end{fact}
	
We remark that the equality $\pd(\NE[G])=\pd(\NE_{\red}[G])$ holds for every graph $G$ by Fact~\ref{fact:remove}. Furthermore, $\NE_{\red}[G]$ is a simple hypergraph without any isolated vertex.
	
A graph $G=(V,E)$ is said to be a \emph{closed neighborhood Sperner graph} (a $\CNS$-graph for short) if there exist no two vertices whose closed neighborhood sets are comparable. Notice that if G is a CNS-graph, then $\NE_{\red}[G]=\NE[G]$.
	
\begin{fact}\label{fact:inci-bip}
$\IN(\NE[G])\cong \BE_e(G)$ for every graph $G$.
\end{fact}
	
Throughout the rest of our work, we abbreviate $\IN(\NE_{\red}[G])$ to $\BE_{red}(G)$, and note that $\BE_{\red}(G)$ is an induced subgraph of $\BE_e(G)$ for every graph $G$, whereas they coincide for $\CNS$-graphs. 
	
\begin{theorem}\label{thm:ineqs} The inequalities
\begin{displaymath}
\alpha(G)\leq \Gamma(G)\leq \pd(\NE[G])\leq \reg(\BE_{\red}(G))\leq \reg(\BE_e(G))\leq \cd(\BE_e(G))
\end{displaymath}
hold for every graph $G$.
\end{theorem}
\begin{proof} As we have already mentioned, the first inequality is satisfied by any graph. Moreover, observe that the equality $i(\NE[G])=|G|-\Gamma(G)$ holds for every graph $G$. Therefore, we conclude that $\Gamma(G)=|G|-i(\NE[G])\leq \pd(\NE[G])$ by Corollary~\ref{cor:helly}.
		
Taking the equality $\pd(\NE[G])=\pd(\NE_{\red}[G])$ into account, the third inequality can be easily derived from~\citep[Theorem $3$]{yusuf} together with Fact~\ref{fact:inci-bip} and Terai's duality~\citep{Terai}. The fourth is due to the fact that $\BE_{\red}(G)$ is an induced subgraph of $\BE_e(G)$, and the last one follows from Lemma~\ref{lem:cochord}.
\end{proof}

We point out that the proof of Theorem~\ref{thm:ineqs} concludes the proof of Theorem~\ref{thm:thm1}. 

\begin{proposition}\label{prop:im-double}
$\im(\BE_e(G))=\IR(G)$ for every graph $G$.
\end{proposition}
\begin{proof}
Let $U=\{u_1,\ldots,u_k\}$ be an irredundant set of maximum size in $G$. 
Therefore, there exists a vertex $v_j\in N_G[u_j]-N_G[U-\{u_j\}]$ for each $1\leq j\leq k$. However, this implies that the set $M(U):=\{u_1v'_1,\ldots,u_kv'_k\}$ is an induced matching in $\BE_e(G)$ so that $\im(\BE_e(G))\geq \IR(G)$. Conversely, suppose that $M=\{x_1y'_1,\ldots,x_py'_p\}$ be an induced matching in $\BE_e(G)$. Notice first that the vertices in $X(M):=\{x_1,\ldots,x_p\}$ must all be distinct. Since $y_i\in N_G[x_i]-N_G[X(M)-\{x_i\}]\neq \emptyset$ for each $1\leq i\leq p$, the set $X(M)$ is an irredundant set in $G$. This proves that $\IR(G)\geq \im(\BE_e(G))$.
\end{proof}
	
By way of examples, we next show that all inequalities of Theorem~\ref{thm:ineqs} could be strict.
\begin{example}
If we consider the $5$-cycle graph $C_5$, we note that
\begin{displaymath}
\alpha(C_5)=\Gamma(C_5)=\IR(C_5)=2<3=\pd(\NE[C_5])=\reg(\BE_e(C_5)),    
\end{displaymath}
where we computed the projective dimension and the regularity in Macaulay$2$~\citep{M2} (over $\Q$).
\end{example}
	
\begin{example}\label{example:gn}
Let $n\geq 3$ be an integer and let $G_n$ be the graph on $V(G_n)=A_n\cup B_n\cup \{u,v\}$, where $A_n=\{a_1,\ldots,a_n\}$ and $B_n=\{b_1,\ldots,b_n\}$ such that 
\begin{align*}
E(G_n):=&\{a_ia_j\colon 1\leq i<j\leq n\}\cup \{a_iu\colon 1\leq i\leq n\}\\
&\cup \{b_ib_j\colon 1\leq i<j\leq n\}\cup \{b_iv\colon 1\leq i\leq n\}\\
&\cup \{a_ib_i\colon 1\leq i\leq n\}.
\end{align*}
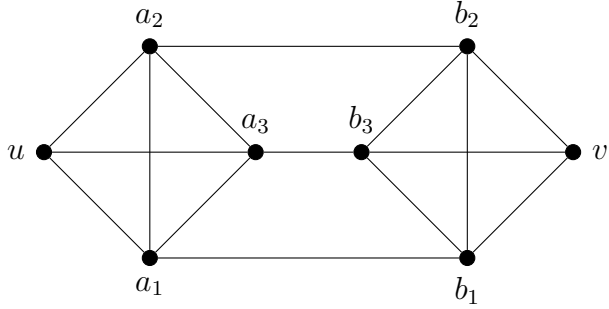
\begin{figure}[ht]
\centering     %%% not \center
\begin{tikzpicture}[scale=0.7]
\node [nod2] at (0,2) (u1) [label=left:$u$]{};
\node [nod2] at (2,0) (u2) [label=below:$a_1$]{}
			edge [] (u1);      
\node [nod2] at (2,4) (u3) [label=above:$a_2$]{} 
			edge [] (u1)
			edge [] (u2);
\node [nod2] at (4,2) (u4) [label=above:$a_3$]{} 
			edge [] (u1)
			edge [] (u2)
			edge [] (u3);
\node [nod2] at (6,2) (v4) [label=above:$b_3$]{} 
			edge [] (u4);
\node [nod2] at (8,0) (v2) [label=below:$b_1$]{} 
			edge [] (v4)
			edge [] (u2);   
\node [nod2] at (8,4) (v3) [label=above:$b_2$]{} 
			edge [] (v4)
			edge [] (v2)
			edge [] (u3);
\node [nod2] at (10,2) (v1) [label=right:$v$]{} 
			edge [] (v2)
			edge [] (v3)
			edge [] (v4);  
\end{tikzpicture}
\caption{The graph $G_3$ in Example~\ref{example:gn}.}
\label{fig:gn}
\end{figure}\medskip
It then follows that $\IR(G_n)=n>2=\Gamma(G_n)$. Indeed, either of the sets $A_n$ and $B_n$ would be a maximum irredundant set in $G_n$. We next define
a co-chordal covering $\{B_1,\ldots,B_n\}$ of $\BE_e(G)$. We let $B_1$ and $B_2$ be the subgraphs of $\BE_e(G_n)$ defined by
\begin{align*}
&E(B_1):=E(\BE_e(G_n[A_n\cup \{u\}]))\cup \{a_1b_1', b_1a_1'\}\quad\textnormal{and}\\
&E(B_2):=E(\BE_e(G_n[B_n\cup \{v\}]))\cup \{b_2a_2', a_2b_2'\},    
\end{align*}
and we set $B_j:=\BE_e(G_n)[\{a_j,a_j',b_j,b_j'\}]$ for each $3\leq j\leq n$. Since $\im(\BE_e(G_n))=\IR(G_n)=\cd(\BE_e(G_n))=n$, we conclude that $\reg(\BE_e(G_n))=n$.\medskip
		
On the other hand, we claim that $\pd(\NE[G_n])=\reg(\BE_{\red}(G_n))=2$ for every $n\geq 3$. Indeed, since the graph $\BE_{\red}(G_n)$ is the disjoint union of two copies of the star graph $K_{1,n+1}$, we conclude that
$\Gamma(G_n)=2\leq \pd(\NE[G])\leq \reg(\BE_{\red}(G_n))=2$ from which the claim follows.
\end{example}
	
\begin{example}\label{example:reg-ir}
Consider the cubic graph depicted in Figure~\ref{fig:reg-ir}. The relevant domination parameters of $H$ were determined by Cockayne and Mynhardt~\citep{Cockayne1997} as $\gamma(H)=5<\alpha(H)=\Gamma(H)=6<7=\IR(H)$. Observe that $H$ is a $\CNS$-graph so that
$\BE_{\red}(H)=\BE_e(H)$. It turns out that $\pd(\NE[H])=6<7=\IR(H)=\reg(\BE_e(H))$ via Macaulay$2$~\citep{M2} (over $\Q$). 
		
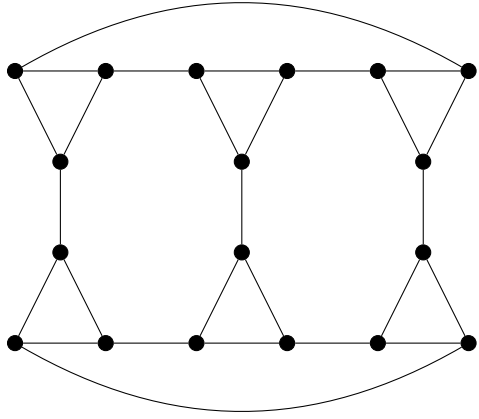
\begin{figure}[ht]
\centering     %%% not \center
\begin{tikzpicture}[scale=0.6]
\node [nod2] at (0,0) (u1) []{};
\node [nod2] at (2,0) (u2) []{}
			edge [] (u1);
\node [nod2] at (4,0) (u3) []{}
			edge [] (u2);                 
\node [nod2] at (6,0) (u4) []{}
			edge [] (u3);
\node [nod2] at (8,0) (u5) []{}
			edge [] (u4);           
\node [nod2] at (10,0) (u6) []{}
			edge [] (u5)
			edge [bend left] (u1);  
\node [nod2] at (1,2) (z1) []{}
			edge [] (u1)
			edge [] (u2);
\node [nod2] at (5,2) (z2) []{}
			edge [] (u3)
			edge [] (u4);                                
\node [nod2] at (9,2) (z3) []{}
			edge [] (u5)
			edge [] (u6);
\node [nod2] at (1,4) (v1) []{}
			edge [] (z1);
\node [nod2] at (5,4) (v2) []{}
			edge [] (z2);
\node [nod2] at (9,4) (v3) []{}
			edge [] (z3);
				
\node [nod2] at (0,6) (w1) []{}
			edge [] (v1);
\node [nod2] at (2,6) (w2) []{}
			edge [] (v1)
			edge [] (w1);
\node [nod2] at (4,6) (w3) []{}
			edge [] (v2)
			edge [] (w2);
\node [nod2] at (6,6) (w4) []{}
			edge [] (v2)
			edge [] (w3);
\node [nod2] at (8,6) (w5) []{}
			edge [] (v3)
			edge [] (w4);
\node [nod2] at (10,6) (w6) []{}
			edge [] (v3)
			edge [] (w5)
			edge [bend right] (w1);
				
\end{tikzpicture}
\caption{The graph $H$ in Example~\ref{example:reg-ir}.}
\label{fig:reg-ir}
\end{figure}\medskip
\end{example}
	
We now discuss about the gap between the Castelnuovo-Mumford regularity and the co-chordal covering number of the extended double cover by demonstrating that it can be arbitrarily large. For a given graph $G$, any edge of $\BE_e(G)$ of the form $v_iv'_i$ is said to be a \emph{diagonal edge}, whereas an edge of the form $v_iv'_j$ for some $i\neq j$ is called a \emph{cross-edge} of $\BE_e(G)$.

\begin{definition}
Let $G=(V,E)$ be a graph. The \emph{conflict-graph} $\CE(G)$ of $G$ is defined to be the graph on the set $\{(u,v)\colon u,v\in V\;\textnormal{and}\;u\in N_G[v]\}$ of ordered pairs such that $(u,v)$ and $(x,y)$ form an edge if and only if $u\notin N_G[y]$ and $x\notin N_G[v]$.
\end{definition}

\begin{proposition}\label{prop:chi-cochord}
$\chi(\overline{G})\leq \chi(\CE(G))\leq \cd(\BE_e(G))$ for every graph $G$.
\end{proposition}
\begin{proof}
The map $v\to (v,v)$ is an isomorphism of $\overline{G}$ onto a subgraph of $\CE(G)$ induced by diagonal pairs $(v,v)$ for $v\in V$. Thus, $\chi(\overline{G})\leq \chi(\CE(G))$.

For the second inequality, let $\RE=\{R_1,\ldots,R_k\}$ be a co-chordal covering of $\BE_e(G)$. If we define $\kappa\colon V(\CE(G))\to [k]$ by $\kappa((x,y)):=\min \{j\colon xy'\in R_j\}$, we claim that $\kappa$ is a proper coloring of $\CE(G)$; hence, $\chi(\CE(G))\leq \cd(\BE_e(G))$.

Suppose that $((u,v),(x,y))\in E(\CE(G))$ is an edge. Then by the definition, we have that $uy',xv'\notin E(\BE_e(G))$. It then follows that cross-edges $uv'$ and $xy'$ witness an induced $2K_2$ in $\BE_e(G)$. Therefore, no subgraph from the family $\RE$ can contain both edges $uv'$ and $xy'$ at the same time. In other words, $\kappa((u,v))\neq \kappa((x,y))$ as claimed.
\end{proof}
%%%%%%%%%%%%%%%%%%%%%%%%%%%%%%%%%%%%%%%%%%%%%%%%%%%%%%%%%%%%%%%%%%%%%%%%%%

\begin{theorem}
For every integer $k\geq 3$, there exists a graph $G_k$ with $\reg(\BE_e(G_k))\leq 3$ and $\cd(\BE_e(G_k))\geq k$.
\end{theorem}
\begin{proof}
Let $H_k$ be a graph with $\girth(H_k)\geq 5$ and $\chi(H_k)\geq k$, and set $G_k:=\overline{H_k}$. By Proposition~\ref{prop:chi-cochord}, we immediately deduce that $\cd(\BE_e(G_k))\geq k$. Thus we are left to prove that $\reg(\BE_e(G_k))\leq 3$. 
		
Let $A,B\subseteq V(G_k)=V(H_k)$ be two arbitrary subsets, and consider the induced bipartite graph $Q:=\BE_e(G_k)[A\cup B']$, where $B':=\{b'\colon b\in B\}$. We note that $ab'\notin E(Q)$ if and only if $a\in N_{H_k}(b)$. We next define a simplicial complex over $A$ by 
\begin{displaymath}
\Gamma_{A,B}:=\{\sigma\subseteq A\colon \exists b\in B\;\textnormal{such that}\;ab'\notin E(Q)\;\textnormal{for every}\;a\in \sigma\}.
\end{displaymath} 
By \citep[Theorem $3.1$]{jonsson}, we conclude that $\Ind(Q)\simeq \Sigma(\Gamma_{A,B})$, where $\Sigma(-)$ denotes the (unreduced) suspension. 
		
Since $H_k$ is $C_4$-free, we have that $|A\cap N_{H_k}(b)\cap N_{H_k}(c)|\leq 1$ for any $b,c\in B$. This in particular implies that any two facets of $\Gamma_{A,B}$ intersect in at most one vertex. However, it then follows that $\widetilde{H}_i(\Gamma_{A,B})=0$ for every $i\geq 2$; hence, $\widetilde{H}_j(\Ind(Q))=0$ for every $j\geq 3$. This proves that $\reg(\BE_e(G_k))\leq 3$.
\end{proof}
%%%%%%%%%%%%%%%%%%%%%%%%%%%%%%%%%%%%%%%%%%%%%%%%%%%%%%%%%%%%%%%%%%%%%%%%%%%%%%%%%%%%
\section{Domination and  matching bounds} 
In this section, we introduce upper bounds to co-chordal covering numbers of extended bipartite double covers of graphs based on matching structures, which will be the main instrument proving some of the cases of Theorem~\ref{thm:thm2}.
	
\begin{definition}
Let $G=(V,E)$ be a graph. A matching $M$ in $G$ is called an \emph{independent matching} if there exists a vertex $x_e\in e$ for each edge $e\in M$ such that the set $X_M:=\{x_e\colon e\in M\}$ is an independent set in $G$. The \emph{independent matching number} $\m_i(G)$ is defined to be the maximum size of an independent matching in $G$. 
\end{definition}

\begin{theorem}\label{thm:cd-mi}
The inequality $\cd(\BE_e(G))\leq |G|-\m_i(G)$ holds for every graph $G$.    
\end{theorem}    
\begin{proof}
Let $M$ be an independent matching in $G$, and let $V(M)=X\cup Y$ be a partition of its vertex set such that $X$ is an independent set in $G$. We write $X=\{x_1,\ldots,x_n\}$ and $Y=\{y_1,\ldots,y_n\}$ so that $M=\{x_1y_1,\ldots,x_ny_n\}$. For every vertex $z\in V\setminus (X\cup Y)$, we consider the double-star subgraph $DS_z$ of $\BE_e(G)$ with the center edge $zz'$. Moreover, we define $DS_i$ to be the subgraph of $\BE_e(G)$ obtained from the double-star centered at the edge $y_iy_i'$ together with the additional edge $x_ix_i'$ (see Figure~\ref{fig:claim1}). Notice that $\{DS_z\colon z\in V\setminus (X\cup Y)\}\cup \{DS_i\colon 1\leq i\leq n\}$ is a co-chordal covering of $\BE_e(G)$. 
\end{proof}
	
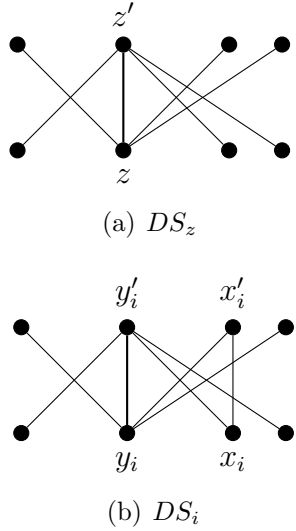
\begin{figure}[ht]
\centering     %%% not \center
{\subfigure[$DS_z$]{
\begin{tikzpicture}[scale=0.7]
\node [nod2] at (0,0) (u1) []{};
\node [nod2] at (0,2) (v1) []{};      
\node [nod2] at (2,0) (z) [label=below:$z$]{} 
			edge [] (v1);
\node [nod2] at (2,2) (zz) [label=above:$z'$]{} 
			edge [] (u1)
			edge [thick] (z);
\node [nod2] at (4,0) (u2) []{} 
			edge [] (zz);
\node [nod2] at (4,2) (v2) []{} 
			edge [] (z);   
\node [nod2] at (5,0) (u3) []{} 
			edge [] (zz);
\node [nod2] at (5,2) (v3) []{} 
			edge [] (z);   
\end{tikzpicture}
}

\hspace*{1mm}
\subfigure[$DS_i$]{
\begin{tikzpicture}[scale=0.7]
\node [nod2] at (0,0) (u1) []{};
\node [nod2] at (0,2) (v1) []{};      
\node [nod2] at (2,0) (z) [label=below:$y_i$]{} 
			edge [] (v1);
\node [nod2] at (2,2) (zz) [label=above:{$y_i'$}]{} 
			edge [] (u1)
			edge [thick] (z);
\node [nod2] at (4,0) (u2) [label=below:$x_i$]{} 
			edge [] (zz);
\node [nod2] at (4,2) (v2) [label=above:{$x_i'$}]{} 
			edge [] (z)
			edge [] (u2);   
\node [nod2] at (5,0) (u3) []{} 
			edge [] (zz);
\node [nod2] at (5,2) (v3) []{} 
			edge [] (z);   
\end{tikzpicture}
}
}
\caption{Illustrations for the proof of Theorem~\ref{thm:cd-mi}.}
\label{fig:claim1}
\end{figure}\medskip
	
Our next example show that the inequality of Theorem~\ref{thm:cd-mi} could be strict.
\begin{example}\label{example:cd-mi}
Let $G$ be the graph depicted in Figure~\ref{fig:cd-mi}. Then, we have that
$\alpha(G)=\Gamma(G)=4<5=\IR(G)=\cd(\BE_e(G))<8=|G|-\m_i(G)$.
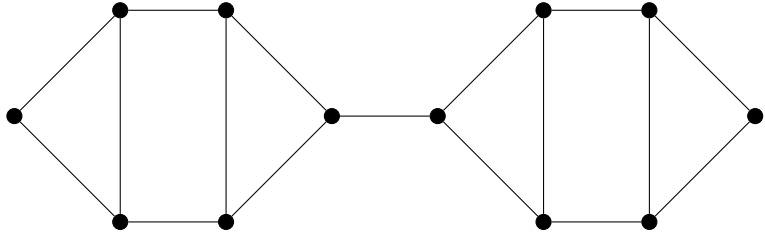
\begin{figure}[ht]
\centering     %%% not \center
\begin{tikzpicture}[scale=0.7]
\node [nod2] at (0,2) (v1) []{};
\node [nod2] at (2,0) (u1) []{}
		edge [] (v1);      
\node [nod2] at (4,0) (u2) []{}
		edge [] (u1); 
\node [nod2] at (6,2) (v2) []{}
		edge [] (u2); 
\node [nod2] at (2,4) (z1) []{}
		edge [] (v1)
		edge [] (u1); 
\node [nod2] at (4,4) (z2) []{}
		edge [] (v2)
		edge [] (u2)
		edge [] (z1);
\node [nod2] at (8,2) (v4) []{}           
		edge [] (v2);
\node [nod2] at (10,0) (u3) []{}
		edge [] (v4);
\node [nod2] at (12,0) (u4) []{}
		edge [] (u3);  
\node [nod2] at (14,2) (v3) []{}
		edge [] (u4);            
\node [nod2] at (10,4) (z3) []{}
		edge [] (v4)
		edge [] (u3);
\node [nod2] at (12,4) (z4) []{}
		edge [] (v3)
		edge [] (u4)
		edge [] (z3);
\end{tikzpicture}
\caption{The graph $G$ in Example~\ref{example:cd-mi}.}
\label{fig:cd-mi}
\end{figure}
\end{example}

Observe that the inequality $\m_i(G)\leq \min\{\alpha(G),\m(G)\}$ holds for every graph $G$. In particular, $\m_i(B)=\m(B)$ for every bipartite graph $B$. We establish a fundamental inequality bridging the domination and independent matching parameters.
	
\begin{theorem}\label{thm:dom-mi}
For every graph $G$ without any isolated vertex, the inequality $\gamma(G)\leq \m_i(G)$ holds.
\end{theorem}
\begin{proof}
Let $G=(V,E)$ be a graph without any isolated vertex. Let $I\subseteq V$ be a maximal independent set of $G$. We construct a bipartite spanning subgraph $B$ of $G$ whose edge set consists precisely of all edges in $E$ that have one endpoint in $I$ and the other in $V\setminus I$.
		
First, we show that $B$ contains no isolated vertices. 
For any $v\in I$, since $G$ has no isolated vertices and $I$ is an independent set, all neighbors of $v$ must belong to $V \setminus I$. Thus, every edge incident to $v$ in $G$ is preserved in $B$, that is, $v$ is not isolated in $B$. 
For any $v\in V\setminus I$, since $I$ is a maximal independent set, $v$ must be adjacent to at least one vertex in $I$. This edge is included in $B$, so $v$ is not isolated in $B$ either.
		
Let $M$ be a maximum matching in the bipartite graph $B$, and let $m=|M|$. By K\"onig's Theorem, $B$ admits a minimum vertex cover $C$ such that $|C|=m$. 
		
We claim that $C$ is a dominating set of $G$. Let $x\in V\setminus C$. Because $B$ has no isolated vertices, the vertex $x$ is incident to at least one edge $e\in E(B)$. Since $C$ is a vertex cover of $B$, it must cover the edge $e$, i.e., the other endpoint of $e$ belongs to $C$. Therefore, $x$ is adjacent to a vertex in $C$. This makes $C$ a dominating set of $B$, and since $B$ is a subgraph of $G$, the set $C$ is also a dominating set of $G$. Consequently, we have that $\gamma(G)\leq |C|=m$.
		
Finally, we show that $M$ is an independent matching in $G$. By the construction of $B$, every edge $e\in M$ has exactly one endpoint in $I$. For each $e\in M$, let $x_e$ be that specific endpoint. The set of representatives $X_M=\{x_e\colon e\in M\}$ is a subset of $I$. Since $I$ is an independent set in $G$, the set $X_M$ is also independent. However, it implies that $M$ is an independent matching. Therefore, the independent matching number $m_i(G)$ must be at least the size of $M$; hence, $m\leq m_i(G)$.
		
As a result, we conclude that $\gamma(G)\leq m\leq m_i(G)$ as claimed.
\end{proof}

We derive the following result as a consequence of Theorems~\ref{thm:cd-mi} and~\ref{thm:dom-mi}.

\begin{corollary}
For every graph $G$, the inequality $\cd(\BE_e(G))\leq |G|-\gamma(G)$ holds.
\end{corollary}

Let $\tau(G)$ denote the vertex cover number of a graph $G$. Graphs satisfying the equality $\gamma(G)=\tau(G)$ have been widely studied in the literature (see, for instance, the comprehensive surveys in~\citep{dombook}).
	
\begin{corollary}\label{cor:dom-vc}
If $G$ is a graph with $\gamma(G)=\tau(G)$, then $\pd(\NE[G])=\reg(\BE_e(G))=\cd(\BE_e(G))=\alpha(G)$.
\end{corollary}
\begin{proof}
Since $\gamma(G)=\tau(G)=|G|-\alpha(G)$, using Theorem~\ref{thm:ineqs}, we conclude that 
\begin{displaymath}
\alpha(G)\leq \pd(\NE[G])\leq \reg(\BE_e(G))\leq \cd(\BE_e(G)) \leq |G|-\gamma(G)=\alpha(G)
\end{displaymath}
from which the result follows.
\end{proof}
	
%%%%%%%%%%%%%%%%%%%%%%%%%%%%%%%%%%%%%%%%%%%%%%%%%%%%%%%%%%%%%%%%%
\section{Proof of Theorem~\ref{thm:thm2}}\label{sec:exact_classes}
	
In this section, we present the proofs for the exact equalities stated in Theorem~\ref{thm:thm2}, structured by the respective graph classes. 

%%%%%%%%%%%%%%%%%%%%%%%%%%%%%%%%%%%%%%%%%%%%%%%%%%%%%%%%%%%%%%%%%%%	
\subsection{K{\"{o}}nig-Egerv{\'{a}}ry and Bipartite graphs}
\begin{definition}
A graph $G=(V,E)$ is said to be a \emph{K{\"{o}}nig-Egerv{\'{a}}ry graph} if $\alpha(G)+\m(G)=|G|$.
\end{definition}

We note that the well-known K{\"{o}}nig-Egerv{\'{a}}ry Theorem implies that every bipartite graph is a K{\"{o}}nig-Egerv{\'{a}}ry graph, but there are also non-bipartite graphs in this class.
	
\begin{theorem}\label{thm:konig_exact}
If $G$ is a K{\"{o}}nig-Egerv{\'{a}}ry graph, then $\pd(\NE[G])=\reg(\BE_e(G))=\cd(\BE_e(G))=\alpha(G)$.   
\end{theorem}
\begin{proof}
Let $S\subseteq V$ be a maximum independent set. Since $G$ is a K{\"{o}}nig-Egerv{\'{a}}ry graph, there exists a maximum matching
$M$ with $V(M)=X\cup Y$ such that $X\subseteq S$ and $Y\subseteq V\setminus S$ by~\citep{levit-2013}. However, this implies that $\m_i(G)=\m(G)$, since $S$ is an independent set. In particular, we have that $\alpha(G)=|G|-\m_i(G)$. By Theorems~\ref{thm:ineqs} and~\ref{thm:cd-mi}, we conclude that
\begin{align*}
\alpha(G)\leq \Gamma(G)\leq \pd(\NE[G])\leq &\reg(\BE_e(G))\\
&\leq \cd(\BE_e(G)) \leq |G|-\m_i(G)=\alpha(G).
\end{align*}
This establishes equality throughout.
\end{proof}
	
\begin{corollary}
If $B$ is a bipartite graph, then 
\begin{displaymath}
\pd(\NE[B])=\reg(\BE_e(B))=\cd(\BE_e(B))=\alpha(B).
\end{displaymath}
\end{corollary}
	
%%%%%%%%%%%%%%%%%%%%%%%%%%%%%%%%%%%%%%%%%%%%%%%%%%%%%%%%%%%%%%
\subsection{Cographs}
We begin recalling that the \emph{join} of two (disjoint) graphs $G_1$ and $G_2$, denoted by $G_1\ast G_2$, is the graph obtained from their  union by adding all possible edges between the vertices of $G_1$ and $G_2$. A graph $G$ is called a \emph{cograph} if it is $P_4$-free. 
		
Cographs have several characterizations, and in particular, they can
be constructed recursively as follows:
\begin{itemize}
\item The single-vertex graph $K_1$ is a cograph.
\item If $G_1$ and $G_2$ are cographs, their disjoint union $G_1\cup G_2$ is a cograph.
\item If $G_1$ and $G_2$ are cographs, their join $G_1\ast G_2$ is a cograph.
\end{itemize}
	
\begin{theorem}\label{thm:cograph_exact}
For every cograph $G$, the equalities
\begin{displaymath}
\pd(\NE[G])=\reg(\BE_e(G))=\cd(\BE_e(G))=\alpha(G)
\end{displaymath}
hold.
\end{theorem}
\begin{proof}
Let $G$ be a cograph. We prove that the inequality $\cd(\BE_e(G))\leq \alpha(G)$  holds for any cograph by using the recursive construction of cographs together with induction.\medskip
		
{\em Base case}: If $G=K_1$, then $\alpha(K_1)=1$. The extended double cover $\BE_e(K_1)$ is a single edge $vv'$, which is trivially co-chordal, providing $\cd(\BE_e(K_1))=1\leq \alpha(K_1)$.\medskip
		
{\em Inductive step} $1$ (the disjoint union): Suppose $G=G_1\cup G_2$ for some cographs $G_1$ and $G_2$. The independence number satisfies $\alpha(G)=\alpha(G_1)+\alpha(G_2)$. Because no edge exists between $G_1$ and $G_2$, the graph $\BE_e(G)$ is the disjoint union of $\BE_e(G_1)$ and $\BE_e(G_2)$. Taking the union of their respective co-chordal covers produces a cover of size $\alpha(G_1)+\alpha(G_2)$.\medskip
		
{\em Inductive step} $2$ (the join operation): Suppose $G=G_1\ast G_2$ for some cographs $G_1$ and $G_2$. We know that $\alpha(G)=\max(\alpha(G_1), \alpha(G_2))$. Assume without loss of generality that $k=\alpha(G_1)\geq \alpha(G_2)$. If $k=1$, the graph $G$ is a complete graph and $\BE_e(G)\cong K_{n,n}$, which is $2K_2$-free. 
		
For $k\geq 2$, let $\{S_1,\dots,S_k\}$ and $\{T_1,\dots,T_{\alpha(G_2)}\}$ be co-chordal covers of $\BE_e(G_1)$ and $\BE_e(G_2)$ respectively. For each $j>\alpha(G_2)$, we set $T_j=\emptyset$. We partition the set of bipartite cross-edges $E_{\textnormal{cross}}$ of $\BE_e(G)$ directionally:
\begin{displaymath}
E_{G_1,G_2}=\{uv'\mid u \in V(G_1), v\in V(G_2)\}, \quad E_{G_2,G_1}=\{ vu'\mid v \in V(G_2), u\in V(G_1)\}.
\end{displaymath}

We next define $k$ subgraphs $U_1,\dots,U_k$ by setting $U_1:=S_1\cup T_1\cup E_{G_1,G_2}$, $U_2:=S_2\cup T_2\cup E_{G_2,G_1}$, and $U_i:= S_i\cup T_i\cup E_{G_1,G_2}$ for $i \geq 3$. 
		
We now show that each subgraph $U_i$ is co-chordal, i.e., no induced $2K_2$ can be sustained within any $U_i$. Let $e_1,e_2\in U_1$. Since $S_1$ and $T_1$ are $2K_2$ free, there are four possibilities of an induced copy of $2K_2$:
\begin{enumerate}
\item $e_1=uu_1^\prime \in S_1$ and $e_2=vv_1^\prime \in T_1$. In this case, we have $uv_1^\prime \in E_{G_1,G_2} \subseteq U_1$. Hence, $e_1, e_2$ can not induce $2K_2$.
			
\item $e_1=uu_1^\prime \in S_1$ and $e_2=u_2v^\prime \in E_{G_1,G_2}$. In this case, we have $uv^\prime \in E_{G_1,G_2} \subseteq U_1$. Hence, $e_1, e_2$ can not induce $2K_2$.
			
\item $e_1=vv_1^\prime \in T_1$ and $e_2=uv^\prime \in E_{G_1,G_2}$. In this case, we have $uv_1^\prime \in E_{G_1,G_2} \subseteq U_1$. Hence, $e_1, e_2$ can not induce $2K_2$.
			
\item $e_1=u_1v_1^\prime, e_2=u_2v_2^\prime \in E_{G_1,G_2}$. In this case, we have $u_1v_2^\prime \in E_{G_1,G_2} \subseteq U_1$. Hence, $e_1, e_2$ can not induce $2K_2$.
\end{enumerate}
A similar argument applies to show that $U_i$ is co-chordal for any $i\geq 2$. This in particular implies that $\cd(\BE_e(G))\leq \alpha(G)$ by the induction.

As a result we conclude that 
\begin{align*}
\alpha(G)\leq \pd(\NE[G])\leq \reg(\BE_e(G))\leq \cd(\BE_e(G))\leq \alpha(G)
\end{align*}
by Theorem~\ref{thm:ineqs}. This completes the proof.
\end{proof}
%%%%%%%%%%%%%%%%%%%%%%%%%%%%%%%%%%%%%%%%%%%%%%%%%%%%%%%%%%%%%%%%	
\subsection{Co-chordal graphs}
We start with the following easy fact, which can be directly obtained from Corollary~\ref{cor:induction-sc} together with an induction. 
\begin{fact}\label{fact:bip-reg}
If $B=(V_1,V_2;E)$ is a bipartite graph, then $\reg(B)\leq \min\{|V_1|,|V_2|\}$.
\end{fact}
		
\begin{lemma}\label{lem:reg-deg-complement}
For any graph $G$ and any vertex $v\in V$, the inequality
\begin{displaymath}
\reg(\BE_e(G))\leq \max \{\reg(\BE_e(G-v)),\deg_{\overline{G}}(v)+1\}
\end{displaymath}
holds.
\end{lemma}
\begin{proof}
As earlier, we write $A':=\{a'\colon a\in A\}$ for any subset $A\subseteq V$, and we do not distinguish a vertex $v$ in $G$ and its corresponding counterpart in $\BE_e(G)$. If we set $B:=\BE_e(G)$, it follows that 
\begin{displaymath}
(V(B)\setminus N_B[v])'=\{x'\colon x\in N_{\overline{G}}(v)\}. 
\end{displaymath}
Therefore, we have that $\reg(B-N_B[v])\leq \deg_{\overline{G}}(v)$ by  Fact~\ref{fact:bip-reg}. As a result, we conclude that
\begin{align*}
\reg(B)&\leq \max \{\reg(B-v),\reg(B-N_B[v])+1\}\\
&\leq \max \{\reg(B-v), \deg_{\overline{G}}(v)+1\}
\end{align*}
by Corollary~\ref{cor:induction-sc}. Now, consider the bipartite graph $B^*:=B-v$. Once again by Fact~\ref{fact:bip-reg}, the inequality $\reg(B^*-N_{B^*}[v'])\leq \deg_{\overline{G}}(v)$ holds, since
\begin{displaymath}
(V(B^*)\setminus N_{B^*}[v'])=N_{\overline{G}}(v). 
\end{displaymath}
However, as $B^*-v'=\BE_e(G-v)$, we obtain that
\begin{displaymath}
\reg(B-v)\leq \max \{\reg(\BE_e(G-v)),\deg_{\overline{G}}(v)+1\}.
\end{displaymath}
It then follows that
\begin{align*}
\reg(B)&\leq \max \{\reg(B-v), \deg_{\overline{G}}(v)+1\}\\
&\leq \max \{\reg(\BE_e(G-v)),\deg_{\overline{G}}(v)+1\}
\end{align*}
as claimed.
\end{proof}
		
\begin{theorem}\label{thm:cochordal_exact}
If $G$ is a co-chordal graph, then $\pd(\NE[G])=\reg(\BE_e(G))=\alpha(G)$.
\end{theorem}
\begin{proof}
Let $v\in V$ be a simplicial vertex of the chordal graph $\overline{G}$. Since $N_{\overline{G}}[v]$ is a clique, we obtain that $\deg_{\overline{G}}(v)+1\leq \omega(\overline{G})=\alpha(G)$. On the other hand, since $G-v$ is also a co-chordal graph, the induction implies that $\reg(\BE_e(G-v))\leq \alpha(G-v)\leq \alpha(G)$. Now, by Lemma~\ref{lem:reg-deg-complement}, we conclude that
\begin{displaymath}
\reg(\BE_e(G))\leq \max \{\reg(\BE_e(G-v)),\deg_{\overline{G}}(v)+1\}\leq \alpha(G),
\end{displaymath}
from which the claim follows.
\end{proof}
		
Lemma~\ref{lem:reg-deg-complement} provides an interesting upper bound that we explain next. Let $G$ be a graph. The \emph{degeneracy} $\degen(G)$ of $G$ is defined to be the smallest integer $k$ such that every nonempty induced subgraph of $G$ has a vertex of degree at most $k$. The following is a direct consequence of Lemma~\ref{lem:reg-deg-complement}.
		
\begin{corollary}
$\pd(\NE[G])\leq \reg(\BE_e(G))\leq \degen(\overline{G})+1$ for any graph $G$.
\end{corollary}
	
\begin{example}
If $G$ is a co-planar graph, then $\pd(\NE[G])\leq \reg(\BE_e(G))\leq 6$.
\end{example}
\begin{example}
If $G$ is a $k$-regular graph on $n$-vertices, then 
\begin{displaymath}
\pd(\NE[G])\leq \reg(\BE_e(G))\leq n-k.
\end{displaymath}
\end{example}
	
%%%%%%%%%%%%%%%%%%%%%%%%%%%%%%%%%%%%%%%%%%%%%%%%%%%%%%%%%%%%%%
\subsection{Chordal graphs}\label{sec:chordal}
Let $G$ be a graph. An ordering $v_1<\ldots<v_n$ of vertices of $G$ is called a \emph{perfect elimination ordering} ($\textnormal{PEO}$ for short) if the vertex $v_i$ is simplicial in $G_i:=G[\{v_i,\ldots,v_n\}]$ for all $1\leq i\leq n$. Recall that a graph is chordal if and only if it has a perfect elimination ordering~\citep[Theorem $1.2.2$]{BLS-gclass}.
	
\begin{theorem}\label{thm:chordal_exact}
For every chordal graph $G$, the equalities
\begin{displaymath}
\pd(\NE[G])=\reg(\BE_e(G))=\cd(\BE_e(G))=\alpha(G)
\end{displaymath}
hold.
\end{theorem}
\begin{proof}
Let $G=(V,E)$ be a chordal graph, and let $v_1<_{\textnormal{PEO}}v_2<_{\textnormal{PEO}}\dots<_{\textnormal{PEO}}v_n$ be a perfect elimination ordering of $G$. 
		
We construct a maximal independent set $I$ greedily along the $\textnormal{PEO}$ by iterating through the vertices from $v_1$ to $v_n$, and adding a vertex to $I$ if it is not adjacent to any previously selected vertex in $I$. Since $G$ is chordal, this greedy approach gives a maximum independent set, that is, $|I|=\alpha(G)=k$. 
Let $I=\{x_1,\dots,x_k\}$.
		
For every vertex $v\in V$, let $f(v)$ be the \textit{earliest} vertex in $I$ (with respect to the $\textnormal{PEO}$) that is in the closed neighborhood $N_G[v]$. Since $I$ is a maximal independent set, $N_G[v]\cap I\neq\emptyset$, making $f(v)$ well-defined for all vertices. 

We partition the vertex set $V$ into $k$ sets $Q_1,\dots,Q_k$, defined by 
$Q_i:=\{v\in V\colon f(v)=x_i\}$.\medskip 
		
{\em Claim} $1$: Every $Q_i$ is a clique in $G$.\medskip

{\em Proof of Claim} $1$: Let $u\in Q_i$. By the definition, $u\in N_G[x_i]$. Suppose, for a contradiction, that $u$ appears before $x_i$ in the $\textnormal{PEO}$ ($u<_{\textnormal{PEO}} x_i$). When the greedy algorithm processed $u$, it did not add $u$ to $I$ (since $u\neq x_i$). This implies $u$ was adjacent to some $x_j\in I$ that appeared before $u$. However, this means $x_j\in N_G[u]$ and $x_j<_{\textnormal{PEO}} u<_{\textnormal{PEO}} x_i$, which contradicts to $f(u)=x_i$. Therefore, every vertex $u\in Q_i$ must appear at or after $x_i$ in the $\textnormal{PEO}$. Since all vertices in $Q_i$ appear at or after $x_i$ and are adjacent to $x_i$, they form a subset of the closed forward neighborhood of $x_i$. By the fundamental property of a $\textnormal{PEO}$, the closed forward neighborhood of any vertex induces a clique. Thus, $Q_i$ is a clique as claimed.\medskip
		
{\em Claim} $2$: The graph $\BE_e(G)$ can be partitioned into $k$ co-chordal subgraphs.\medskip

{\em Proof of Claim} $2$: We partition the edges of $\BE_e(G)$ into $k$ subgraphs $S_1,\dots,S_k$ using the following assignment rules:
\begin{enumerate}
\item For any vertex $v \in V$, assign the edge $vv'$ to $S_{f(v)}$.
\item For any edge $uv \in E(G)$, exactly one of $u<_{\textnormal{PEO}} v$ or $v <_{\textnormal{PEO}} u$ holds. Assume $u<_{\textnormal{PEO}} v$. We assign both corresponding cross-edges, $uv'$ and $vu'$, to $S_{f(u)}$. 
\end{enumerate}
This partitions the edges of $\BE_e(G)$. We now prove that each $S_i$ is $2K_2$-free. 
		
Let $e_1=ab'$ and $e_2=cd'$ be any two edges in $S_i$. By our assignment rules, there exists a vertex $u_1\in \{a,b\}$ such that $u_1\in Q_i$ and $u_1 \leq_{\textnormal{PEO}} a,b$ (if $a=b$, then $u_1=a$). Similarly, there exists a vertex $u_2\in \{c,d\}$ such that $u_2\in Q_i$ and $u_2 \leq_{\textnormal{PEO}} c,d$. 
		
Since $Q_i$ is a clique, it follows that $u_1\in N_G[u_2]$. Without loss of generality, we may assume  that $u_1\leq_{\textnormal{PEO}} u_2$. 
By the definition of a $\textnormal{PEO}$, the closed forward neighborhood of $u_1$ forms a clique. Since $u_1\leq_{\textnormal{PEO}} a$ and $u_1\leq_{\textnormal{PEO}} b$ (and $u_1\in \{a,b\}$), both $a$ and $b$ belong to the closed forward neighborhood of $u_1$. Furthermore, since $u_1 \leq_{\textnormal{PEO}} u_2$ and $u_1\in N_G[u_2]$, the vertex $u_2$ also belongs to the closed forward neighborhood of $u_1$. 
		
Consequently, the set of vertices $\{a,b,u_2\}$ must induce a clique in $G$. Since $u_2\in \{c,d\}$, we have two cases to analyze.\medskip

{\em Case} $1$: $u_2=c$. Then the set $\{a,b,c\}$ is a clique so that $c\in N_G[b]$. By the definition of $\BE_e(G)$, we conclude that $cb'\in E(\BE_e(G))$.\medskip

{\em Case} $2$: $u_2=d$. Then the set $\{a,b,d\}$ is a clique so that $d \in N_G[a]$. Once again, by the definition of $\BE_e(G)$, we obtain that $ad'\in E(\BE_e(G))$.

In all possible cases, there exists an edge between $\{a,b'\}$ and $\{c,d'\}$ in $\BE_e(G)$. Therefore, $e_1$ and $e_2$ cannot form an induced $2K_2$. 
		
Since every $S_i$ is $2K_2$-free, the graph $\BE_e(G)$ is covered by $k=\alpha(G)$ co-chordal subgraphs, proving that $\cd(\BE_e(G))\leq \alpha(G)$. This completes the proof by Theorem~\ref{thm:ineqs}.
\end{proof}
	
\begin{remark}
Theorem~\ref{thm:chordal_exact} could also be proved using homological arguments by making use of independence complex concept. Since this idea of proof could be of independent interest, we provide it in Section~\ref{sec:append} (Appendix).
\end{remark}
%%%%%%%%%%%%%%%%%%%%%%%%%%%%%%%%%%%%%%%%%%%%%%%%%%%%%%%%%
\subsection{Comparability graphs}\label{sec:comparability}
We begin with recalling the definition of comparability graphs~\citep{BLS-gclass}. A  graph $G=(V,E)$ is called a \emph{comparability graph} if there exists a partial order $\leq$ on $V$ such that $xy\in E$ if and only if either $x\leq y$ or $y\leq x$ in $P$ for any two distinct vertices $x,y\in V$.

\begin{theorem}\label{thm:comparability}
For every comparability graph $G$, the equalities
\begin{displaymath}
\pd(\NE[G])=\reg(\BE_e(G))=\cd(\BE_e(G))=\alpha(G)
\end{displaymath}
hold.
\end{theorem}
\begin{proof}
Let $P=(V,\leq)$ be a poset (partially ordered set) whose comparability graph is $G$. Since the independent sets of $G$ correspond to the anti-chains of $P$, 
the width of $P$ is $k=\alpha(G)$. Recall that the bipartite graph $\BE_e(G)$ has bipartition $V\cup V'$ such that
\begin{displaymath}
xy'\in E(\BE_e(G)) \quad\Longleftrightarrow\quad x\leq y\;\textnormal{or}\;y\leq x.
\end{displaymath}
By Dilworth's theorem, the associated poset $P$ can be partitioned $V=C_1\cup\cdots \cup C_k$ into $k$ chains of $P$. For each $i\in [k]$, we define
a spanning subgraph $H_i$ of $\BE_e(G)$ by
\begin{displaymath}
E(H_i)=\bigcup_{\substack{c\in C_i,\ x\in V\\ c\leq x}}\{cx',\ xc'\}.
\end{displaymath}
We claim first that these subgraphs cover $E(\BE_e(G))$. Indeed, for any edge
$xy'$, vertices $x$ and $y$ are comparable in $P$, say $x\leq y$. Then, if $x\in C_i$, we conclude that $xy'\in E(H_i)$. Furthermore, observe that $E(H_i)$ also contains the diagonal edge $xx'$ for every vertex $x\in V$.

We next claim that every $H_i$ is a co-chordal graph. Let $i\in [k]$ be given, and write $C=C_i$. Consider two vertex-disjoint edges of $H_i$.
There are three cases to separate.

\begin{enumerate}
\item The edges have the form $cx'$ and $dy'$, where $c,d\in C$, $c\leq x$, and $d\leq y$ in $P$. Since $C$ is a chain, after interchanging the edges if
necessary, we may assume $c\leq d$. Then, $c\leq d\leq y$ so that $cy'\in E(H_i)$.

\item The edges have the form $xc'$ and $yd'$, where $c,d\in C$, $c\leq x$, and $d\leq y$ in $P$. Once again, assume $c\leq d$. Then $c\leq d\leq y$, so
$yc'\in E(H_i)$.

\item The edges have the form $cx'$ and $yd'$, where $c,d\in C$, $c\leq x$, and $d\leq y$ in $P$. Since $c$ and $d$ lie in the same chain, either
$c\leq d$ or $d\leq c$. In either case, the definition of $H_i$ implies that $cd'\in E(H_i)$.
\end{enumerate}

Therefore, in every case, a cross-edge prevents the two edges from inducing a $2K_2$ in $H_i$ as claimed.
 
As a consequence, we obtain that
\begin{displaymath}
\cd(\BE_e(G))\leq k=\alpha(G).
\end{displaymath}
Now, the result follows from Theorem~\ref{thm:ineqs}.
\end{proof}

%%%%%%%%%%%%%%%%%%%%%%%%%%%%%%%%%%%%%%%%%%%%%%%%%%%%%%%%%%%
\section{Further Comments}
We note that all graph classes (except possibly non-bipartite K{\"{o}}nig-Egerv{\'{a}}ry graphs) in Theorem~\ref{thm:thm2} constitute a subclass of strongly perfect graphs. We may therefore ask whether the equality $\cd(\BE_e(G))= \alpha(G)$ holds for the broad class of strongly perfect graphs, since $\alpha(G)=\Gamma(G)=\IR(G)$ for such graphs. We demonstrate that this is in fact false via the following counterexample. 
	
\begin{example}\label{example:cd-alpha}
Consider the graph $G$ depicted in Figure~\ref{fig:cd-alpha}, which is constructed by taking a $6$-cycle $C_6$ on the vertex set $\{1,2,\dots,6\}$
(with edges in cyclic fashion) and adding a vertex $z$ adjacent to vertices $1,2,4,5$. It is straightforward to verify that $\alpha(G)=\chi(\overline{G})=3$. Notice that $G$ is a $\CNS$-graph, meaning a reduction does not alter its closed neighborhood hypergraph. Computations via Macaulay2~\cite{M2} confirm that $\reg(\BE_e(G))=3$ (over $\Q$). To prove that $\cd(\BE_e(G)) \geq 4$, we leverage the conflict-graph formulation established in Proposition~\ref{prop:chi-cochord}. This time, computations via SageMath~\cite{sagemath} confirm that $\chi(\CE(G))=4$. Therefore, we conclude that $\cd(\BE_e(G))\geq 4>3=\alpha(G)$. This graph serves as a counterexample showing that $\cd(\BE_e(G))$ can strictly exceed $\alpha(G)$ even in strongly perfect graphs. 
\begin{figure}[ht]
\centering     %%% 
\begin{tikzpicture}[scale=0.7]
\node [nod2] at (0,2) (v1) [label=left:$3$]{};
\node [nod2] at (0,4) (v2) [label=left:$2$]{}
           	edge [] (v1);
\node [nod2] at (2,6) (v3) [label=above:$1$]{}
           	edge [] (v2);         	      
\node [nod2] at (4,4) (v4) [label=right:$6$]{}
           	edge [] (v3);
\node [nod2] at (4,2) (v5) [label=right:$5$]{}
           	edge [] (v4);
\node [nod2] at (2,0) (v6) [label=below:$4$]{}
           	edge [] (v5)
           	edge [] (v1);
\node [nod2] at (2,3) (v7) [label=right:$z$]{}
           	edge [] (v2)
           	edge [] (v3)
           	edge [] (v5)
           	edge [] (v6);         	         	         	
\end{tikzpicture}
\caption{The graph $G$ in Example~\ref{example:cd-alpha}.}
\label{fig:cd-alpha}
\end{figure}
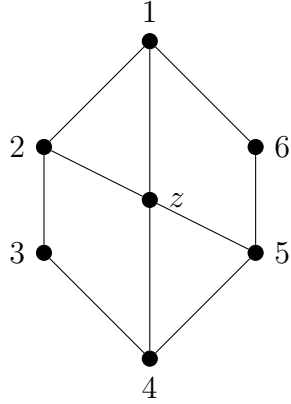
\end{example}
	
Therefore, the following seems to be a more natural question in this direction.
\begin{question}
Is it true that for any strongly perfect graph $G$, the equality $\reg(\BE_e(G))=\alpha(G)$ holds?
\end{question}

We remark that when passing from strongly perfect graphs to their superclass of perfect graphs, the equalities guaranteed by Theorem~\ref{thm:dom-eq} are no longer available. However, we believe that known structural properties of perfect graphs are still strong enough to provide a clear picture in the computations of relevant algebraic parameters. 

\begin{question}\label{ques:perfect}
Characterize all perfect graphs $G$ such that the equality
$\reg(\BE_e(G))=\IR(G)$ hold?
\end{question} 
We note that the upper irredundance number $\IR(G)$ is nothing but the induced matching number of the extended double cover $\BE_e(G)$	by Proposition~\ref{prop:im-double}. In other words, Question~\ref{ques:perfect} basically asks for which perfect graphs $G$, the regularity of $\BE_e(G)$ coincides with its induced matching number.

%%%%%%%%%%%%%%%%%%%%%%%%%%%%%%%%%%%%%%%%%%%%%%%%%%%%%%%%%%%%%%
%%%%%%%%%%%%%%%%%%%%%%%%%%%%%%%%%%%%%%%%%%%%%%%%%%%%%%%%%%
\section*{Declarations}
\noindent\textbf{Data Availability Statement:} Data sharing is not applicable to this article as no new data were created or analyzed in this study.\medskip

\noindent\textbf{Conflict of Interest:} The authors declare that they have no relevant financial or non-financial interests to disclose.\medskip

\noindent\textbf{Use of AI:} No AI tools were used in generating the mathematical content or proofs.

%%%%%%%%%%%%%%%%%%%%%%%%%%%%%%%%%%%%%%%%%%%%%%%%%%%%%%%%%%%
\bibliographystyle{abbrv}  
%\bibliography{cnh}

%%%%%%%%%%%%%%%%%%%%%%%%%%%%%%%%%%%%%%%%%%%%%%%%%%%%%%%
\section*{Appendix: A Topological Proof for Chordal graphs}\label{sec:append}
For chordal graphs, in Theorem~\ref{thm:chordal_exact}, we gave a combinatorial edge-routing proof of the exact equalities $\pd(\NE[G])=\reg(\BE_e(G))=\alpha(G)$. As a complement to this proof, we present an alternative homological approach to bounding the regularity of extended double covers for chordal graphs. This method relies heavily on the structure of independence complexes and the Mayer-Vietoris sequence.
	
Recall that an edge $e=xy$ of a bipartite graph $B=(X,Y;E)$ is said to be a \emph{bisimplicial edge} if $N_B[e]:=N_B[x]\cup N_B[y]$ is a biclique of $B$.
	
\begin{fact}
If $v$ is a simplicial vertex of $G$, then $vv'$ is a bisimplicial edge of $\BE_e(G)$.
\end{fact}  
	
Let $\Delta$ be a simplicial complex and let $\Delta_1$ and $\Delta_2$ be complexes such that $\Delta=\Delta_1\cup \Delta_2$. Then, the following sequence is exact:
	
\begin{displaymath}
\cdots\to \widetilde{H}_i(\Delta_1\cap \Delta_2)\to \widetilde{H}_i(\Delta_1)\oplus \widetilde{H}_i(\Delta_2)\to \widetilde{H}_i(\Delta)\to \widetilde{H}_{i-1}(\Delta_1\cap \Delta_2)\to\cdots
\end{displaymath}
	
The following decomposition result is based on an application of Hochster’s formula together with the above Mayer-Vietoris exact sequence, for which we note that a restriction to an induced subcomplex commutes with union and intersection. 
	
\begin{proposition}\label{prop:reg-union}
If $\Delta=\Delta_1\cup \Delta_2$ is a union of simplicial complexes, then
\begin{displaymath}
\reg(\Delta)\leq \max \{\reg(\Delta_1),\reg(\Delta_2), \reg(\Delta_1\cap \Delta_2)+1\}.
\end{displaymath}
\end{proposition}

\begin{lemma}\label{lem:bisimplicial}
Let $B=(X,Y;E)$ be a bipartite graph. If $e=xy$ is a bisimplicial edge of $B$, then
\begin{displaymath}
\reg(B)=\max \{\reg(B-\{x,y\}), \reg(B-N_B[e])+1\}.
\end{displaymath}
\end{lemma}
\begin{proof}
Since $e=xy$ is bisimplicial, we observe that
\begin{align*}
&\Ind(B)=(x\ast \Ind(B-N_B[x]))\cup (y\ast \Ind(B-N_B[y]))\\
&\textnormal{and}\\
&(x\ast \Ind(B-N_B[x]))\cap (y\ast \Ind(B-N_B[y]))=\Ind(B-N_B[e]),
\end{align*}
where $z\ast \Delta$ denotes the cone of a simplicial complex $\Delta$ with apex $z$. Indeed, let $S$ be an independent set in $B$. If $S$ contains $x$ or $y$, then it belongs to the corresponding cone complex. Otherwise, it can not intersect both $N_B(x)\setminus \{y\}$ and $N_B(y)\setminus \{x\}$, since $e=xy$ is bisimplicial. If it avoids $N_B(x)\setminus \{y\}$, then $S\cup \{y\}$ is an independent set in $B$. Similarly, if it avoids $N_B(y)\setminus \{x\}$, then $S\cup \{x\}$ is an independent set in $B$. This proves the union equality, and a similar argument applies to the intersection. 
		
Recall that $\reg(z\ast \Delta)=\reg(\Delta)$ for any simplicial complex $\Delta$. Since $\Ind(B-N_B[x])$ and $\Ind(B-N_B[y])$ are induced subcomplexes of $\Ind(B-\{x,y\})$, we conclude that
\begin{displaymath}
\max\{\reg(\Ind(B-N_B[x])),\reg(\Ind(B-N_B[y]))\}\leq \reg(\Ind(B-\{x,y\})).
\end{displaymath}
Therefore, by Proposition~\ref{prop:reg-union}, the following inequalities
\begin{align*}
\reg(B)&\leq \max\{\reg(B-N_B[x]), \reg(B-N_B[y]), \reg(B-N_B[e])+1\}\\
&\leq \max \{\reg(B-\{x,y\}), \reg(B-N_B[e])+1\}
\end{align*}
hold.	
		
For the converse, we note that if we write $R:=V((B-N_B[e]))\cup \{x,y\}$, then $B[R]\cong (B-N_B[e])\cup K_2$ is an induced subgraph of $B$; hence, $\reg(B[R])=\reg(B-N_B[e])+1\leq \reg(B)$. Similarly, since $B-\{x,y\}$ is also an induced subgraph of $B$, we obtain that 
\begin{displaymath}
\reg(B)\geq \max \{\reg(B-\{x,y\}), \reg(B-N_B[e])+1\},
\end{displaymath}
which completes the proof.
\end{proof}
	
\begin{corollary}\label{cor:simplicial}
If $v$ is a simplicial vertex of $G$, then
\begin{displaymath}
\reg(\BE_e(G))=\max \{\reg(\BE_e(G-v)),\reg(\BE_e(G-N_G[v]))+1\}.
\end{displaymath}
\end{corollary}
\begin{proof}
Since $v$ is a simplicial vertex, the diagonal edge $vv'$ is bisimplicial in $\BE_e(G)$. Thus, the claim follows from Lemma~\ref{lem:bisimplicial} together with
$\BE_e(G)-\{v,v'\}=\BE_e(G-v)$ and $\BE_e(G)-N_{\BE_e(G)}[vv']=\BE_e(G-N_G[v])$.
\end{proof}
	
\begin{theorem}\label{thm:hom-chordal}
If $G$ is a chordal graph, then
\begin{displaymath}
\pd(\NE[G])=\reg(\BE_e(G))=\alpha(G).
\end{displaymath}
\end{theorem}
\begin{proof}
We already know the (in)equalities
\begin{displaymath}
\alpha(G)=\Gamma(G)=\IR(G)=\im(\BE_e(G))\leq \pd(\NE[G])\leq \reg(\BE_e(G))
\end{displaymath}
hold for a chordal graph $G$ by Theorems~\ref{thm:dom-eq} and~\ref{thm:ineqs}. Now, let $v$ be a simplicial vertex of $G$. By Corollary~\ref{cor:simplicial} together with an induction, we conclude that
\begin{align*}
\reg(\BE_e(G))&=\max \{\reg(\BE_e(G-v)),\reg(\BE_e(G-N_G[v]))+1\},\\
&\leq \max \{\alpha(G-v),\alpha(G-N_G[v])+1\},\\
&\leq \alpha(G).
\end{align*}
\end{proof}
%%%%%%%%%%%%%%%%%%%%%%%%%%%%%%%%%%%%%%%%%%%%%%%%%%%%%%%
\end{document}